\documentclass[12pt, reqno]{amsart}

\usepackage{amsmath, amsthm, amssymb, stmaryrd}
\usepackage{hyperref}
\usepackage{enumerate}
\usepackage{url}

\usepackage{xcolor}

\usepackage[centering, margin={1in, 0.5in}, includeheadfoot]{geometry}

\newtheorem{thm}{Theorem}[section]
\newtheorem{lemma}[thm]{Lemma}

\newtheorem{cor}[thm]{Corollary}
\newtheorem{conj}[thm]{Conjecture}

\theoremstyle{definition}

\theoremstyle{remark}
\newtheorem{remark}[thm]{Remark}

\numberwithin{equation}{section}

\def\alp{{\alpha}} 
\def\bet{{\beta}}  
\def\gam{{\gamma}} 
\def\del{{\delta}} 
\def\eps{\varepsilon}

\def\le{\leqslant}

\def\ge{\geqslant}
\def\geq{\geqslant}

\def \bN {\mathbb N}

\def \bQ {\mathbb Q}
\def \bR {\mathbb R}
\def \bZ {\mathbb Z}

\def \cA {\mathcal A}
\def \cB {\mathcal B}

\def \cE {\mathcal E}
\def \cF {\mathcal F}
\def \cG {\mathcal G}

\def \cK {\mathcal K}

\def \cS {\mathcal S}

\def \dim {\mathrm{dim}}

\def \Bad {{\mathrm{\mathbf{Bad}}}}
\def \dimH {{\mathrm{dim}_{\mathrm{H}}}}

\begin{document}
\title[Inhomogeneous multiplicative approximation in Bad]{Fully-inhomogeneous multiplicative diophantine approximation of badly approximable numbers}
\author{Sam Chow \and Agamemnon Zafeiropoulos}
\address{Mathematics Institute, Zeeman Building, University of Warwick, Coventry CV4 7AL, United Kingdom}
\email{sam.chow@warwick.ac.uk}
\address{Department of Mathematical Sciences, Norwegian University of Science and Technology, 7491 Trondheim, Norway}
\email{agamemnon.zafeiropoulos@ntnu.no}

\begin{abstract}
We establish a strong form of Littlewood's conjecture with inhomogeneous shifts, for a full-dimensional set of pairs of badly approximable numbers on a vertical line. We also prove a uniform assertion of this nature, generalising a strong form of a result of Haynes, Jensen and Kristensen. Finally, we establish a similar result involving inhomogeneously badly approximable numbers, making progress towards a problem posed by Pollington, Velani, Zafeiropoulos and Zorin.
\end{abstract}
\maketitle

\section{Introduction}

A famous, long-standing open problem in Diophantine approximation is {\it Littlewood's conjecture}, which states that if $\alpha, \beta\in \bR$ then
\begin{equation} \label{LC}
    \liminf_{n\to\infty} n\|n\alpha\| \cdot \|n\beta\| = 0,
\end{equation}  
where for $x \in \bR$ we write $\| x \| = \min \{ |x-k|: k \in \bZ \}$. Observe that \eqref{LC} holds trivially unless both $\alpha$ and $\beta$ belong to the set 
\[\Bad = \{ \alp \in \bR : \liminf_{n\to\infty} n\|n\alpha\|>0 \}\]
of badly approximable numbers. As such, we may regard the conjecture as a statement about pairs of badly approximable numbers. Badly approximable numbers and Littlewood's conjecture arise naturally in dynamical systems via bounded orbits. Moreover, Margulis \cite{Mar2000} made a highly-influential conjecture in homogeneous dynamics that generalises Littlewood's.

Refining a celebrated result of Pollington and Velani \cite{PV}, it was shown in \cite[Equation (18)]{PVZZ2021} that if $\alp \in \Bad$ and $\del \in \bR$ then there exists $\cG = \cG(\alp,\del) \subseteq \Bad$ of full Hausdorff dimension such that if $\bet \in \cG$ then
\[
\# \{ n \in [N]: n \| n \alp \| \cdot \| n \bet - \del \| \le 1/ \log n \} \gg \log \log N,
\]
for large $N \in \bN$. We extend this further to the fully-inhomogeneous setting. In addition, we relax the condition $\alp \in \Bad$ to $\alp \in \cK$, where
\begin{equation} \label{Kdef}
\cK = \{ \alp \in \bR \setminus \bQ: \sup\{ t^{-1} \log q_t(\alp): t \in \bN \} < \infty \}.
\end{equation}
Here $q_1(\alp),q_2(\alp),\ldots$ are the continued fraction convergent denominators of $\alp$, see \S \ref{SequenceSection}.

\begin{thm} \label{MainThm}
Let $\alp \in \cK$ and $\gam, \del \in \bR$. Then there exists $\cG  = \cG(\alp,\gam,\del) \subseteq \Bad$ of Hausdorff dimension $\dimH(\cG) = 1$ such that if $\bet \in \cG$ then
\begin{equation} \label{Counting}
\# \{ n \in [N]: n \| n \alp - \gam \| \cdot \| n \bet - \del \| \le 1/ \log n \} \gg \log \log N,
\end{equation}
for large $N \in \bN$.
\end{thm}

The fact that $\Bad \subset \cK$ follows from badly approximable numbers having bounded partial quotients \cite[Theorem 1.4]{BRV2016}, and from the basic recursion \eqref{DenomRecurs}. Note also that we have imposed an explicit, generic Diophantine condition on $\alp$. Indeed, the set $\cK$ has full Lebesgue measure, in the sense that the Lebesgue measure of its complement $\bR \setminus \cK$ is zero. This follows immediately from a result of Khintchine \cite{Khi1936}, refined by L\'evy \cite{Lev1936} who showed that
\[
\frac{\log q_t(\alp)}{t} \to \frac{\pi^2}{12 \log 2}
\]
for almost all $\alp$, see \cite[Chapter V]{RS1992}.

Our approach involves a sequence of probability measures constructed by Kaufman \cite{Kau1980}, which enjoy a certain Fourier decay property, see also \cite{JS2016, QR2003}. The idea is to study the Diophantine approximation rate of a number that is generic with respect to this measure, meaning that the results should hold outside of a set of measure zero. This is philosophy in the field of \emph{metric Diophantine approximation}. Lebesgue measure is not suitable for studying badly approximable numbers in this way, since $\Bad$ has Lebesgue measure zero. If one works with pairs of real numbers, without restricting attention to badly approximable numbers, then the corresponding Lebesgue measure problems have been keenly investigated by Gallagher \cite{Gal1962} and subsequent authors \cite{BHV2020, Cho2018, CT2019, CT2021, CL2021}.

It would be desirable to have a version of \eqref{Counting} involving numbers that are badly approximable in an inhomogeneous sense. To clarify what is meant here, for $\del \in \bR$ we define
\[
\Bad(\del) = \{ \bet \in \bR: 
\liminf_{n \to \infty} n \| n \bet - \del \| > 0 \}.
\]

\begin{conj} \label{FirstConj}
Let $\alp \in \cK$ and $\gam, \del \in \bR$. Then there exists $\cG = \cG(\alp,\gam,\del) \subseteq \Bad(\del)$ of Hausdorff dimension $\dimH(\cG) = 1$ such that if $\bet \in \cG$ then
we have \eqref{Counting}
for large $N \in \bN$.
\end{conj}

Conjecture \ref{FirstConj} generalises a problem posed in \cite{PVZZ2021}. We expect that proving results of this type would involve the construction of \emph{inhomogeneous Kaufman measures}, that is, measures analogous to Kaufman's that are supported on $\Bad(\del)$, where $\del \in \bR$ is arbitrary.

In addition to the theory surrounding Kaufman's measures, we require a second main input. A sequence $n_1,n_2,\ldots$ in $(0,\infty)$ is \emph{lacunary} if for some $c > 1$ we have 
\[
n_{i+1} \ge cn_i
\qquad (i \in \bN).
\]
Given $\alp \in \cK$ and $\gam \in \bR$ we construct a lacunary sequence $n_1,n_2, \ldots$ of positive integers, admitting a polynomial upper bound, such that
\[
\| n_t \alp - \gam \| \ll n_t^{-1} \qquad (t \in \bN).
\]
The point is that this mimics the sequence of continued fraction denominators of $\alp$, adapting to the inhomogeneous shift. Enthusiasts may wonder if this sequence enables a fully-inhomogeneous version of a uniform inhomogeneous approximation result of Haynes and friends \cite{HJK2014}, which uses discrepancy theory \cite{KN1974}. Below we state the slightly stronger version from \cite{TZ2020}.

\begin{thm} [Technau--Zafeiropoulos 2020, improving Haynes--Jensen--Kristensen 2014]
Fix $\eps > 0$, and let $\alp_1,\alp_2,\ldots$ be a sequence of badly approximable numbers. Then there exists $\cG \subseteq \Bad$ of Hausdorff dimension $\dimH(\cG) = 1$ such that if $\bet \in \cG$, $i \in \bN$ and $\del \in \bR$ then
\[
n \| n \alp_i \| \cdot \| n \bet - \del \| < \frac{(\log \log \log n)^{\eps + 1/2}}{(\log n)^{1/2}}
\]
has infinitely many solutions $n \in \bN$.
\end{thm}

Indeed it does.

\begin{thm} \label{thm2}
Fix $\eps > 0$, let $\alp_1,\alp_2,\ldots$ be a sequence in $\cK$, and let $\gam_1,\gam_2,\ldots$ be a sequence of real numbers. Then there exists $\cG \subseteq \Bad$ of Hausdorff dimension $\dimH(\cG) = 1$ such that if $\bet \in \cG$, $i \in \bN$ and $\del \in \bR$ then
\begin{equation} \label{seven}
n \| n \alp_i - \gam_i \| \cdot \| n \bet - \del \| < \frac{(\log \log \log n)^{\eps + 1/2}}{(\log n)^{1/2}}
\end{equation}
has infinitely many solutions $n \in \bN$.
\end{thm}

\noindent
We stress that the set $\cG$ in Theorem \ref{thm2} does not depend on $\del$, which is why the rate of approximation is poorer than in \eqref{Counting}. Improving the exponent attached to the logarithm is an open problem with some community interest \cite{HJK2014, PVZZ2021}.

Theorem \ref{thm2} adds to the list of uniform inhomogeneous approximation results obtained so far. The most notable such result is due to Shapira \cite{Shap2008}, who proved that almost all pairs $(\alp,\bet)\in\bR^2$ satisfy 
\[\liminf_{n\to\infty}n\|n\alp-\gam\| \cdot \|n\bet-\del \|=0 \quad \text{ for all }\gam,\del \in\bR,  \]
answering a question of Cassels on whether such pairs $(\alp, \bet)$ even exist. This was subsequently improved by Gorodnik and Vishe \cite{GV2016}, who showed that there exists $\varepsilon>0$ such that almost all $(\alp, \bet)\in \bR^2 $ satisfy 
\[\liminf_{n\to\infty}(\log_5 n)^{\varepsilon}n\|n\alp-\gam\| \cdot \|n\bet-\del \|=0 \quad \text{ for all }\gam,\del \in\bR,  \]
where $\log_5$ denotes the fifth iterate of the function $x\mapsto \max\{1, \log x\}.$  Theorem \ref{thm2} shares with these results of Shapira and Gorodnik--Vishe the feature of uniformity in the shift $\del$. It differs from these results in that its focus is on badly approximable $\bet$, and that it lacks uniformity in the shifts $\gam_i$.
Note that we attain a stronger approximation rate, and that we are able to fix the $\alp_i$.
This stronger approximation rate comes not from working with badly approximable numbers, but rather from non-uniformity in the shifts $\gam_i$. In fact, the Lebesgue analogue of Theorem \ref{thm2} follows straightforwardly by mimicking its proof, since Lebesgue measure on $[0,1]$ has polynomial Fourier decay rate \cite[Equation (4)]{TZ2020}, and by periodicity we only need to consider $\bet \in [0,1]$. The outcome is as follows.

\begin{thm} \label{thm2alt}
Fix $\eps > 0$, let $\alp_1,\alp_2,\ldots$ be a sequence in $\cK$, and let $\gam_1,\gam_2,\ldots$ be a sequence of real numbers. Then, for almost all $\bet \in \bR$, if $i \in \bN$ and $\del \in \bR$ then \eqref{seven}
has infinitely many solutions $n \in \bN$.
\end{thm}

Exploiting the uniformity in $\del$ in Theorem \ref{thm2}, we are able to make progress towards Conjecture \ref{FirstConj}, albeit in a weakened form. Setting
\[
\cB = \{ (\bet, \del) \in \bR^2: \bet \in \Bad \cap \Bad(\del) \},
\]
we have the following Corollary of Theorem \ref{thm2}.

\begin{cor} \label{thm2cor}
Fix $\eps > 0$, let $\alp_1,\alp_2,\ldots$ be a sequence in $\cK$, and let $\gam_1,\gam_2,\ldots$ be a sequence of real numbers. Then there exists $\cE \subseteq \cB$ of Hausdorff dimension $\dimH(\cE) = 2$ such that if $i \in \bN$ and $(\bet,\del) \in \cE$ then the inequality \eqref{seven} has infinitely many solutions $n \in \bN$.
\end{cor}

Though Corollary \ref{thm2cor} is closer to Theorem \ref{thm2} than to Theorem \ref{MainThm},
its salient feature is that it involves inhomogeneously badly approximable numbers.  Corollary \ref{thm2cor} is \emph{doubly metric} in nature, see \cite[Footnote 2]{ET2011}. Kleinbock \cite{Kle1999} previously showed, by dynamical means, that
$\dim(\cB) = 2$, though that result is more general. Corollary \ref{thm2cor} can alternatively be viewed as a refinement of a special case of Kleinbock's result.

\subsection*{Organisation} We construct our sequence of inhomogeneous approximations in Section \ref{SequenceSection}. We then prove Theorems \ref{MainThm} and \ref{thm2} in Sections \ref{Proof1} and \ref{Proof2}, respectively. Finally, we establish Corollary \ref{thm2cor} in \S \ref{DoublyMetricSection}.

\subsection*{Notation} Given $x\in\bR$ we write $\|x\|=\min\{|x-k|: k\in\bZ\}$ for the distance of $x$ from its nearest integer. When $N\geq 1$ is a positive integer, we write $[N] = \{1,2,\ldots,N\}$ for the set of positive integers at most equal to $N$.  We denote by $\dim_{\mathrm{H}}(\cA)$ the Hausdorff dimension of a subset $\cA\subseteq \bR$. For $s > 0$, we write $H^s$ for Hausdorff $s$-measure. We adopt the Vinogradov and Bachman--Landau notations: if functions $f$ and $g$ output non-negative real values, we write $f \ll g$, $g \gg f$, or $f = O(g)$ if $f \le Cg$ for some constant $C$, and we write $f \asymp g$ if $f \ll g \ll f$.

\subsection*{Funding} SC was supported by EPSRC Fellowship Grant EP/S00226X/2, and by a Mittag-Leffler Junior Fellowship. AZ is supported by a postdoctoral fellowship funded by Grant 275113 of the Research Council of Norway. 

\subsection*{Acknowledgements}
We thank Victor Beresnevich and Evgeniy Zorin for helpful conversations.

\section{A sequence of inhomogeneous approximations}
\label{SequenceSection}

Let $\alp \in \bR \setminus \bQ$, and let
\[
\alp = [a_0; a_1, a_2, \ldots]
\]
be the continued fraction expansion of $\alp$, where the \emph{partial quotients} are $a_0 \in \bZ$ and $a_1,a_2,\ldots \in \bN$. Let
$p_0/q_0, p_1/q_1, \ldots$ be the convergents of the continued fraction expansion of $\alp$. Specifically, define
\[
p_{-1} = 1, \qquad p_0 = a_0,
\qquad p_k = a_k p_{k-1} + p_{k-2} \quad (k \in \bN)
\]
and
\begin{equation} \label{DenomRecurs}
q_{-1} = 0, \qquad q_0 = 1,
\qquad q_k = a_k q_{k-1} + q_{k-2} \quad (k \in \bN).
\end{equation}
Though the sequence $q_1,q_2,\ldots$ may or may not be lacunary, the subsequence $q_2,q_4,q_6,\ldots$ is lacunary, since
\[
q_k = a_k q_{k-1} + q_{k-2} \ge 2 q_{k-2} \qquad (k \in \bN).
\]
It then follows by induction that 
\[
q_{2t} \ge 2^t \qquad (t \in \bN).
\]

Recalling \eqref{Kdef}, we now come to the main agenda item for this section. The key features of the sequence constructed below are that it is lacunary and that
\[
\log n_t \asymp t,
\qquad n_t \| n_t \alp - \gam \| \ll 1
\qquad (t \in \bN).
\]

\begin{lemma} \label{sequence} Let $\alp \in \cK$, and let $q_0,q_1,q_2,\ldots$ be the continued fraction denominators of $\alp$. Put
\[
C = \sup \{ t^{-1} \log q_t: t \in \bN \} \in (0,\infty),
\]
and let $\gam \in \bR$. Then there exists a lacunary sequence $n_1, n_2, \ldots$ of positive integers such that
\begin{equation} \label{ineqs}
8^t < n_t \le 4e^{6Ct},
\quad
\| n_t \alp - \gam \|
\le 8/n_t \qquad (t \in \bN).
\end{equation}
\end{lemma}

\begin{proof} Let $t \in \bN$. We will choose $n_t = q_{6t} + b$, where
\begin{equation} \label{bineq}
1 \le b \le m := 2q_{6t} + q_{6t-1} - 1, \qquad \| b \alp - \gam \| \le 1/q_{6t}.
\end{equation}
Observe from the theory of continued fractions that
\[
\| q_{6t} \alp \| \le 1/q_{6t+1} \le 1/q_{6t}.
\]

To choose $b$, we employ the three distance theorem in the form \cite{MK1998}. By \cite[Corollary 1]{MK1998}, wherein $k=6t$, $r=1$ and $s=q_k-1$, if we write the fractional parts
\[
\{ \alp \}, \{ 2 \alp \}
\ldots, \{ m \alp \}
\]
as $d_1,\ldots,d_m$ with the ordering
\[
0 =: d_0 < d_1 < \cdots < d_m < d_{m+1} := 1,
\]
then
\[
\max \{ d_{i+1} - d_i: 0 \le i \le m \}
\le a_{6t+1}/q_{6t+1}\le 1/q_{6t}.
\]
As $\{ \gam \} \in [d_i, d_{i+1}]$ for some $i$, there must exist $b \in [m]$ such that 
\[
\| b \alp - \gam \| \le 1/q_{6t}.
\]
This confirms \eqref{bineq}.

We now have 
\[
8^t \le q_{6t} < n_t \le 4q_{6t} \le 4e^{6Ct}
\]
for $t \in \bN$, and by the triangle inequality
\[
\| n_t \alp - \gam \| \le 2/q_{6t} \le 8/n_t.
\]
The sequence is lacunary, for if $t \in \bN$ then
\[
n_{t+1} > q_{6t+6} \ge 8q_{6t} \ge 2n_t.
\]
\end{proof}

\section{Kaufman's measures}
\label{Proof1}

In this section, we establish Theorem \ref{MainThm}. Let $M \in \bZ_{\ge 3}$, and denote by $\cF_M$ the set of
\[
[0;a_1,a_2,\ldots]
\]
with $a_1,a_2,\ldots \in [M]$. Note that $\cF_M \subseteq \Bad$, by the standard characterisation of badly approximable numbers as irrationals with bounded partial quotients \cite[Theorem 1.4]{BRV2016}. In pioneering work, Kaufman \cite{Kau1980} constructed a probability measure $\mu$ supported on $\cF_M$ whose Fourier transform decays as
\[
\hat \mu(t) \ll (1+|t|)^{-7/10^4}.
\]

\begin{remark} Owing to the work of Queff\'elec and Ramar\'e \cite{QR2003}, one can take $M=2$ here and still have a polynomial Fourier decay rate. We do not require this case.
\end{remark}

With $C$ as in Lemma \ref{sequence}, there exists a lacunary sequence $n_1, n_2, \ldots$ of positive integers satisfying \eqref{ineqs}, which implies
\[
\log n_t \asymp t \qquad (t \in \bN).
\]
Define $\psi: \bN \to [0,1]$ by $\psi(1)=0$ and
\[
\psi(n) = 1/(8 \log n) \qquad (n\ge 2).
\]
By \cite[Theorem 1]{PVZZ2021}, we have
\[
\# \{ t \in [T]: \| n_t \bet - \del \| \le \psi(n_t) \} = 2 \Psi(T)
+ O(\Psi(T)^{2/3} (\log \Psi(T) + 2)^{2.1})
\]
for $\mu$-almost all $\bet \in \cF_M$, where
\[
\Psi(T) = \sum_{t \le T} \psi(n_t) \asymp \log T
\qquad (T \ge 2).
\]
Observing that
\[
n_t \|n_t \alp - \gam\| \cdot \|n_t \bet - \del\| \le  1/\log n_t
\]
whenever $\|n_t \bet - \del\| \le \psi(n_t)$, we now have
\[
\# \{ t \in [T]:
n_t \|n_t \alp - \gam\| \cdot \|n_t \bet - \del\| \le 1/\log n_t \}
\gg \log T,
\]
for $\mu$-almost all $\bet \in \cF_M$.

Finally, if $N \in \bN$ is large then 
\[
n_T \le 4e^{6CT} < N \le 4e^{6C(T+1)}
\]
for some $T \in \bN$. Therefore
\[
\# \{ n \in [N]: n \| n \alp - \gam\| \cdot \| n \bet - \del\| \le 1/\log n \} \\
\gg
\log T \gg \log \log N,
\]
for $\mu$-almost all $\bet \in \cF_M$. It then follows from \cite[Remark 7]{PVZZ2021} that the Hausdorff dimension of the set of $\bet \in \Bad$ satisfying \eqref{Counting} is greater than or equal to $\dim_{\mathrm{H}}(\cF_M)$.
To finish the proof of Theorem \ref{MainThm}, recall Jarn\'ik's conclusion \cite{Jar1928} that
\[
\lim_{M \to \infty} \dim_{\mathrm{H}}(\cF_M) = 1.
\]

\section{A uniform result}
\label{Proof2}

In this section, we adapt \cite{TZ2020} to establish Theorem \ref{thm2}. Denote by $\cG$ the set of $\bet \in \Bad$ such that if $i \in \bN$ and $\del \in \bR$ then \eqref{seven} has infinitely many solutions $n \in \bN$. 
Choose an integer $M \ge 3$  and let $\mu$ be Kaufman's measure on $\cF_M$.

Let $i \in \bN$, and let $n_1,n_2,\ldots$ be the sequence obtained by applying Lemma \ref{sequence} to $\alp_i$ and $\gam_i$. Let $B = B_i > 0$, and put
\[
\psi(1) = \psi(2) = 0,
\qquad
\psi(T) = B^{-1} T^{-1/2} (\log \log T)^{\eps+1/2}
\quad (T \in \bZ_{\ge 3}).
\]
Arguing as in \cite[Section 3]{TZ2020}, we find that for $\mu$-almost all $\bet$, for all $\del \in \bR$ we have a well defined, increasing sequence given by
\[
T_k = T_k(\bet,\del) = \min \{ 
T \in \bN:
\# \{ t \in [T]: 
\| n_t \bet - \del \| < \psi(T) \} = k\} \qquad (k \in \bN).
\]
Thus, by \eqref{ineqs}, for $\mu$-almost all $\bet$ and all $\del \in \bR$ we have 
\begin{align*}
    n_{T_k}\|n_{T_k}\alp_i-\gam_i\| \cdot \|n_{T_k}\bet - \del \| \ll \frac{(\log\log T_k)^{\eps+1/2}}{B T_k^{1/2}}
    \ll  \frac{ (\log\log\log n_{T_k})^{\eps + 1/2} }{B(\log n_{T_k})^{1/2}}.
\end{align*}
Choosing $B$ sufficiently large yields
\[
n_{T_k}\|n_{T_k}\alp_i-\gam_i\| \cdot \|n_{T_k}\bet - \del \| <
  \frac{ (\log\log\log n_{T_k})^{\eps + 1/2} }{(\log n_{T_k})^{1/2}}
  \qquad (k \in \bN).
\]

The upshot is that for $\mu$-almost all $\beta\in\Bad$, if $i\in \bN$ and $\del\in\bR$ then \eqref{seven} holds for infinitely many $n \in \bN$. In other words, we have
\[
\mu(\cG) = 1.
\]
Now according to \cite[Remark 7]{PVZZ2021} we have $\dim_{\mathrm{H}}(\cG) \ge \dim_{\mathrm{H}}(\cF_M)$, and since $M$ can be chosen arbitrarily large we conclude that $\dim_{\mathrm{H}}(\cG)=1.$

This finishes the proof of Theorem \ref{thm2}. As discussed in the introduction, Theorem \ref{thm2alt} is almost identical; one replaces $\mu$ by Lebesgue measure on $[0,1]$.

\section{A doubly metric problem}
\label{DoublyMetricSection}

Corollary \ref{thm2cor} follows readily from Theorem \ref{thm2} and two further ingredients.

\subsection{Twisted Diophantine approximation}

Tseng \cite{Tse2009} established that if $\bet \in \bR$ then $\dimH(\cS_\bet) = 1$, where
\begin{equation} \label{TwistedBad}
\cS_\bet = \{ \del \in \bR: \bet \in \Bad(\del) \}.
\end{equation}
These types of results are referred to as \emph{twisted Diophantine approximation} statements, see \cite{BM2017, Har2012}. In fact, Tseng proved \emph{a fortiori} that $\cS_\bet$ is \emph{winning} in the sense of Schmidt \cite{Sch1966}.

\subsection{Marstrand's slicing theorem}

Marstrand's slicing theorem \cite[Theorem 5.8]{Fal1986} is a seminal result in fractal geometry. 

\begin{thm} [Marstrand's slicing theorem] \label{Marstrand}
Let $\cE \subseteq \bR^2$, and let $\cG \subseteq \bR$. For $\bet \in \bR$, set
\[
\cE_\bet = \{ (\bet, \del) \in \cE: \del \in \bR \}.
\]
Let $s,t \in [0,1]$, let $c > 0$, and suppose that
\[
H^t(\cE_\bet) > c \qquad (\bet \in \cG).
\]
Then
\[
H^{s+t}(\cE) \gg c H^s(\cG),
\]
and the implied constant only depends on $s$ and $t$.
\end{thm}

We now proceed in earnest towards Corollary \ref{thm2cor}. Let $s \in (0,1)$, and let $\cG$ be as in Theorem \ref{thm2}. For $\bet \in \bR$, let 
\[
\cE_\bet = \{ \bet \} \times \cS_\bet,
\]
where $\cS_\bet$ is as in \eqref{TwistedBad}. Finally, put
\[
\cE = \bigcup_{\bet \in \cG} \cE_\bet.
\]
By Tseng's theorem, we have $H^s(\cE_\bet) = H^s(\cS_\bet) = \infty$ for all $\bet$, and by Theorem \ref{thm2} we have $H^s(\cG) = \infty$.
By Marstrand's slicing theorem, Theorem \ref{Marstrand}, we therefore have
\[
H^{2s}(\cE) = \infty,
\]
and we conclude that $\dimH(\cE) = 2$.

\providecommand{\bysame}{\leavevmode\hbox to3em{\hrulefill}\thinspace}

\end{document}